\documentclass[12pt, reqno]{amsart}
\input xypic
\xyoption{all}
\usepackage{lipsum}
\newcounter{quotecount}
\newcommand{\MyQuote}[1]{\vspace{0.3cm}\refstepcounter{quotecount}%
     \parbox{14cm}{\em #1}\hspace*{0.5cm}(\Roman{quotecount})\\[0.3cm]}

\newcommand{\mult}{{\rm mult}}

\newcommand{\PP}{{\mathbb P}}

\newcommand{\tor}{\xymatrix{\ar@{-->}[r]&}}

\begin{document}

\title[ ]{On Cremona Transformations of $\PP^3$ with all possible bidegrees}

\newtheorem{thm}{Theorem}
\newtheorem{pro}[thm]{Proposition}
\newtheorem{cor}[thm]{Corollary}
\newtheorem{rems}[thm]{Remarks}
\newtheorem{lem}[thm]{Lemma}
\newtheorem{defi}{Definition}

\theoremstyle{remark}
\newtheorem{exa}[thm]{Example}
\newtheorem{exas}[thm]{Examples}

\newtheorem{rem}[thm]{Remark}

\author{Ivan Pan}\footnote{Partially supported by the \emph{Agencia Nacional de Investigadores} of Uruguay}
%\address{Ivan Pan, Centro de Matemática, Facultad de Ciencias, Universidad de la Rep\'ublica, Igu\'a 4225, 11400 - Montevideo - URUGUAY}
\email{ivan@cmat.edu.uy}
\maketitle

\begin{abstract}
For every orderer pair $(d,e)$ of integer numbers $d,e\geq 2$, such that $\sqrt{d}\leq e\leq d$,  we construct a birational map $\PP^3\tor\PP^3$ defined by homogeneous polynomials of degree $d$ whose  inverse map is defined by homogeneous polynomials of degree $e$.
\end{abstract}

\section{Introduction}\label{sec1}

The aim of this note is to correct a mistake in the proof of  Theorem \cite[Th\'eor\`eme. 2.2]{Pa-multi}. The proof of that theorem depends on the example \cite[Exemple 2.1]{Pa-multi} which is wrong. 

We propose an explicit construction of Cremona transformations of $\PP^3$ (see \S\,\ref{sec2}, especially Lemma \ref{lem1}) which, together with their inverse maps, provide all possible bidegrees (Theorem \ref{thm2} and Corollary \ref{cor3}).

\bigskip

\noindent{\bf Acknowledge} We would like to thank Igor Dolgachev for pointing out a mistake in \cite[Exemple 2.1]{Pa-multi}. 
\section{Main construction and results}\label{sec2}   

Let $\PP^3$ be the projective space over an algebraically closed field $k$ of characteristic zero; we fix  homogeneous coordinates $w,x,y,z$ on $\PP^3$. 

We recall that a Cremona transformation of $\PP^3$ is a birational map $F:\PP^3\tor\PP^3$. We say $F$ has \emph{bidegree} $(d,e)$ when $F$ and its inverse $F^{-1}$ are defined by homogeneous polynomials, without non trivial common factors, of degrees $d$ and $e$ respectively; notice that in this case $F^{-1}$ has bidegree $(e,d)$. If $V\subset \PP^3$ is a dense open set over which $F^{-1}$ is defined and injective and $L\subset\PP^3$ is a line with $L\cap V\neq\emptyset$, then $e$ is the degree of the closure of $F^{-1}(L\cap V)$; one deduces that $\sqrt{d}\leq e\leq d$  (see for example \cite[\S 1]{Pa-multi}).

If $X\subset \PP^2$ is a curve and $p\in \PP^2$ we denote by $\mult_p(X)$ the multiplicity of $X$ at $p$. If $S,S'\subset\PP^3$ are surfaces and $C\subset S\cap S'$ is an irreducible component, we denote by $\mult_C(S,S')$ the intersection multiplicity of $S$ and $S'$ along $C$.

\smallskip

Consider a rational map $T:\PP^3\tor\PP^3$ defined by
\[T=(g:qt_1:qt_2:qt_3),\]
where $t_1,t_2,t_3\in k[x,y,z]$ are homogeneous of degree $r$, without non trivial common factors, and $g,q\in k[w,x,y,z]$ are homogeneous of degrees $d,d-1$, with $d\geq r\geq 1$ and $g$ irreducible.  We know that $T$ is birational if $\tau:=(t_1:t_2:t_3):\PP^2\tor\PP^2$ is birational and $g,q$ vanish at $o=(1:0:0:0)$ with orders $d-1$ and $\geq d-r-1$, respectively (see \cite[Proposition 2.2]{Pa-stell}).

On the other hand, consider $2r-1$ points $p_0, p_1,\ldots,p_{2r-2}$ in $\PP^2$, $r\geq 2$, satisfying the following condition:

\MyQuote{
 There exist curves $X_r,Y_{r-1}\subset\PP^2$ of degrees $r,r-1$, respectively, with $X_r$ irreducible, such that $\mult_{p_0}(X_r)=r-1$,  $\mult_{p_0}(Y_{r-1})\geq r-2$ and $p_i\in X_r\cap Y_{r-1}$ for $i=1,\ldots,2r-2$.}\label{quote}   

Hence \emph{loc.\,cit.} also implies there exists a plane Cremona transformation defined by polynomials of degree $r$ with a point of multiplicity $r-1$ at $p_0$ and passing through $p_1,\ldots,p_{2r-2}$ with multiplicity 1: indeed, if we consider $p_0=(1:0:0)$ and take polynomials $t_1$ and $f$, of degrees $r$ and $r-1$,  defining $X_r$ and $Y_{r-1}$ respectively, then $(t_1:yf:zf):\PP^2\tor\PP^2$ is a Cremona transformation as required; such a  transformation is said to be \emph{associated} to the points $p_0, p_1,\ldots,p_{2r-2}$.     
 
\begin{rem}
The transformations satisfying the condition (I) are general cases of the so-called \emph{de Jonqui\`eres transformations}  (see \cite{dJo} or \cite[Def. 2.6.10]{Alb}). We note that the Enriques criterion \cite[Thm. 5.1.1]{Alb} may be  used to prove that a set of $2r-2$ points $p_0, p_1,\ldots,p_{2r-2}$ with assigned multiplicities $r-1,1,\ldots,1$, and satisfying the condition (I), defines a de Jonqui\`eres transformation. 
\end{rem}

Set $r=d$ and take an irreducible homogeneous polynomial $g=wA(x,y,z)+B(x,y,z)$ of degree $d$; that is, $q\in k-\{0\}$ in the considerations above. Denote by $T_{g,\tau}$ the Cremona transformation defined by
\begin{equation}\label{cretra}
T_{g,\tau}=(g:t_1:t_2:t_3),
\end{equation}
where $\tau=(t_1:t_2:t_3)$ is associated to $2d-1$ points satisfying the condition (I).

We have

\begin{lem}\label{lem1}
Let $d\geq 2$ be an integer number.  Then

(a)  there exist $g$ and $\tau$ such that $T_{g,\tau}$ has bidegree $(d, 2d-1-m)$, for $0\leq m\leq d-1$.

(b) there exist $g$ and $\tau$ such that $T_{g,\tau}$ has bidegree $(d, d^2-\ell^2-m)$, for $0\leq \ell<d-1$ and $0\leq m\leq 2d-2$.

\end{lem}

\begin{proof}
We identify $\PP^2$ with the plane $\{w=0\}\subset\PP^3$ and consider a point $p_0\in\PP^2$. Without loss of generality, we may suppose $p_0=(0:1:0:0)$. We recall $o=(1:0:0:0)$.

In order to prove (a) we first choose $g\in k[w,x,y,z]$ to be a homogeneous polynomial which vanishes along the line $op_0$ with order $d-1$ and is general with respect to this condition. In other words, one has $g=wA+B$ with 
\[A=A_{d-1}(y,z), B=xB_{d-1}(y,z)+B_{d}(y,z),\]
where $A_i, B_i\in k[y,z]$ are general homogeneous polynomials of degree $i$. Hence $A=0$ defines a union of $d-1$ distinct lines in $\PP^2$ passing through $p_0$ and $B=0$ defines an irreducible curve of degree $d$ with an ordinary singular point of multiplicity $d-1$ at $p_0$.

Notice that, by construction, in the open set $\PP^2-\{p_0\}$ the curves $A=0$ and $B=0$ intersect at $d(d-1)-(d-1)^2=d-1$ points;  in particular, if  $m\leq d-1$,  there exist $m$ points $p_1,\ldots,p_{m}\in\PP^2$ satisfying $A(p_i)=B(p_i)=0$ for $1\leq i\leq m$. We consider $m$ such points and choose $2d-1-m$ points  $p_{m+1},\ldots,p_{2d-2}\in\PP^2$ with $A(p_j)\neq 0$ and  $B(p_j)=0$, for all $j=m+1,\ldots,2d-2$, such that $p_0,p_1,\ldots,p_{2d-2}$ satisfy (I). Let $\tau$ be a plane Cremona transformation associated to these $2d-1$ points.  

Now we consider a Cremona transformation  $T_{g,\tau}:\PP^3\tor\PP^3$ as in (\ref{cretra}). A general member in the linear system defining $T_{g,\tau}$ is an irreducible surface of degree $d$, $S$ say, with equation of the form
\[ag+a_1t_1+a_2 t_2+a_3t_3=0,\]
where $a,a_1,a_2,a_3\in k$ are general. Therefore $S$ admits an ordinary singularity of multiplicity $d-1$ at the generic point of (the line) $op_0$ and is smooth at the generic point of $op_i$ for $1\leq i\leq m$. If $S'$ is another general member of that linear system, then there exists an irreducible rational curve $\Gamma$ of degree $e=\deg(T_{g,\tau}^{-1})$ such that the intersection scheme $S\cap S'$ is supported on
\[\Gamma \cup\left(\cup_{i=0}^{m} op_i\right).\]
We have 
\[\mult_{\Gamma}(S,S')=1, \mult_{op_0}(S,S')=(d-1)^2, \mult_{op_i}(S,S')=1, i=1,\ldots, m, \]
hence $e=d^2-(d-1)^2-m=2d-1-m$, which proves the assertion (a). 

To prove (b) we proceed analogously. This time we choose $g=wA+B$ with 
\[A=\sum_{i=\ell}^{d-1}x^{d-1-i}A_{i}(y,z), B=\sum_{j=\ell}^{d}x^{d-j}B_{j}(y,z),\]
where  $A_i, B_i\in k[y,z]$ are general homogeneous polynomials of degree $i$.  Since $\ell\leq d-2$ there exist points $p_1,\ldots,p_{2d-2}\in\PP^2$ such that $A(p_i)=B(p_i)=0$ for  $1\leq i\leq m$ and $A(p_j)\neq 0$ $B(p_j)=0$ for $j=m+1,\ldots,2d-2$: indeed, in the open set $\PP^2-\{p_0\}$, the curves $A=0$ and $B=0$ intersect at $d(d-1)-\ell^2\geq d(d-1)-(d-2)^2=3d-4$ points. Thus we can define $\tau$ as before and obtain the assertion (b).

\end{proof}

\begin{thm}\label{thm2}
There exist Cremona transformations of bidegree $(d,e)$ for $d\leq e\leq d^2$. 
\end{thm}

\begin{proof}
From the  part (a) of Lemma \ref{lem1} we deduce that there exist Cremona transformations of bidegrees $(d,e)$ for $d\leq e\leq 2d-1$. 

Now we use the part (b) of Lemma \ref{lem1}. Suppose $\ell<d-1$ and think of  $e=d^2-\ell^2-m$ as a function $e(\ell,m)$ depending on $\ell, m$; to complete the proof it suffices to show that the  image of that function contains $\{2d,2d+1,\ldots,d^2\}$. 

We note that $e(d-2,2d-2)=2d-2$ and $e(0,0)=d^2$; in other words, the part (b) of Lemma \ref{lem1} implies that there exist Cremona transformations of bidegrees $(d,2d-2)$ and $(d,d^2)$. On the other hand $e(\ell,0)-e(\ell-1,2d-2)=2(d-\ell)-1>0$. Since $e(\ell,m)$ decreases with respect to $m$, we easily obtain the result. 
\end{proof}

For $d=2$ the theorem above asserts that there exist Cremona transformations of bidegrees $(2,2), (2,3), (2,4)$; analogously for $d=3$ and  bidegrees $(3,3), (3,4), \ldots, (3,9)$, and so on. By symmetry we deduce

\begin{cor}\label{cor3}
There exist Cremona transformations of bidegrees $(d,e)$ with $\sqrt{d}\leq e\leq d^2$.
\end{cor}

\begin{rem}
 The inequality $\sqrt{d}\leq e\leq d^2$ is the unique obstruction to the degree for the inverse of a Cremona transformation of degree $d$ in $\PP^3$.
\end{rem}

\end{document}